\documentclass[reqno, 11pt]{amsart}

\usepackage{color, comment}

\pdfoutput=1
\usepackage{graphicx}
\usepackage{amssymb, amsmath, amsthm}
 \usepackage[bookmarks=false]{hyperref}
 \usepackage{mathscinet}
\allowdisplaybreaks

\makeatletter
\renewcommand{\@biblabel}[1]{#1.} 
\makeatother

\usepackage{hyperref}
\hypersetup{
    colorlinks=true, 
    linktoc=all,     
    linkcolor=blue,  
    citecolor=blue
}

\newtheorem{Theorem}{Theorem}[section]
\newtheorem{Proposition}[Theorem]{Proposition}
\newtheorem{Lemma}[Theorem]{Lemma}

\theoremstyle{definition}

\theoremstyle{remark}
\newtheorem{remark}[Theorem]{Remark}
\newtheorem*{Remarks}{Remarks}
\newtheorem*{Remark}{Remark}

\newcommand{\suml}{\sum\limits}

\newcommand{\yvec}{{y_1,\dots,y_n}}

\newcommand{\sumN}{{\left| \boldsymbol{N} \right|}}

\newcommand{\sumvec}[2]{#1_1+\cdots + #1_#2} 
  
\newcommand\sumK{{\left| \boldsymbol {K} \right|}}
\newcommand{\summ}{{\left| \boldsymbol {m} \right|}}
\newcommand\sumj{{\left| \boldsymbol {j} \right|}}

\newcommand\sumk{{\left| \boldsymbol {k} \right|}}

\newcommand{\N}{{\boldsymbol N}}
\newcommand{\y}{{\boldsymbol y}}

\newcommand{\F}{{\boldsymbol F}}
\newcommand{\G}{{\boldsymbol G}}
\newcommand{\M}{{\boldsymbol M}}
\newcommand{\D}{{\boldsymbol D}}

\renewcommand{\k}{{\boldsymbol k}}
\renewcommand{\j}{{\boldsymbol j}}
\newcommand{\m}{{\boldsymbol m}}

\newcommand{\kvec}{{k_1,\dots,k_n}}

\newcommand{\multsum}[3]{{\sum\limits_{\substack{{0\le #1_#3 \le #2_#3} \\  
{#3 =1,2,\dots, n}}}}}

\newcommand{\qrfac}[2]{{\left({#1}\right)_{#2}}} 

\newcommand{\pqrfac}[3]{{\left({#1};#3\right)_{#2}}}

\newcommand{\triprod}[1]{\prod\limits_{1\le r < s \le #1}}

\newcommand{\sqprod}[1]{\prod\limits_{r, s =1}^{#1}} 

\newcommand{\smallprod}[1]{\prod\limits_{r =1}^{#1}} 

\newcommand{\xover}[1]{{#1_{r}}/{#1_{s}}}

\newcommand{\powerq}[2]{q^{\suml_{r=1}^{#2} (r-1)#1_r }} 

\newcommand{\vandermonde}[3]{\triprod{#3} \! 
\frac{1-q^{#2_r-#2_s} \xover {#1} }{1-\xover{#1}}
}

\newcommand{\multisum}[2]{\underset{r=1, \dots, #2}{\sum\limits_{#1_r\geq 0}}} 

\newmuskip\pFqskip
\mathchardef\pFcomma=\mathcode`, 

\newcommand*\rPhis[6]{%
  \begingroup
  \begingroup\lccode`~=`,
    \lowercase{\endgroup\def~}{\pFcomma\mkern\pFqskip}%
  \mathcode`,=\string"8000
  {}_{#1}\phi_{#2}\Biggl[\genfrac..{0pt}{}{#3}{#4};#5, #6\Biggr]%
  \endgroup
}

\usepackage{hyperref}
\hypersetup{
    colorlinks=true, 
    linktoc=all,     
    linkcolor=blue,  
    citecolor=blue,
    urlcolor=blue
}

\numberwithin{equation}{section}

\begin{document} 
\title[Expansion formulas for multiple series]{Expansion formulas for multiple basic hypergeometric series over root systems}
\author[G.~Bhatnagar]{Gaurav Bhatnagar
}

\address{Ashoka University, Sonipat, Haryana 131029, India}
\email{bhatnagarg@gmail.com}

\author[S.~Rai]{Surbhi Rai}
\address{Department of Mathematics, Indian Institute of Technology, Delhi 110067, India.}
\email{rai.surbhi2@gmail.com }

\date{\today}




\begin{abstract}
We extend expansion formulas of Liu given in 2013 to the context of multiple series over root systems. Liu and others have shown the usefulness of these formulas in Special Functions and number-theoretic  contexts. We extend Wang and Ma's generalizations of Liu's work which they obtained using $q$-Lagrange inversion.  
We use the  $A_n$ and $C_n$ Bailey transformation and other summation theorems due to Gustafson, Milne, Milne and Lilly, and others, from the theory of $A_n$, $C_n$ and $D_n$ basic hypergeometric series. 
\end{abstract}

\maketitle 


\section{Introduction}\label{sec:intro}
In a series of papers, Zhi-Guo Liu 
 extended some of the central summation and transformation formulas of basic hypergeometric series. In particular, Liu extended Rogers' non-terminating very-well-poised $_6\phi_5$ summation formula, Watson's transformation formula, and gave an alternate approach to the orthogonality of the Askey-Wilson polynomials; see Liu~\cite{Liu2013a, Liu2013b, Liu2013c, Liu2014} and Liu and Zeng~\cite{LiuZeng2015}. These have been shown to be useful  in number-theoretic contexts too;  in particular, Chan and Liu~\cite{ChanLiu2018}, Wang and Yee~\cite{WangYee2019}, Wang~\cite{Wang2019} and Chen and Wang~\cite{ChenWang2020} have used them to prove Hecke-type identities, and identities for false theta  and mock theta functions.  All this work relies on three expansion formulas of Liu. It is our intent here to provide several infinite families of extensions of Liu's key formulas to multiple basic hypergeometric series over root systems.

But first we give an exposition of Liu's key formulas, and their generalization by Jin Wang and Xinrong Ma~\cite{WM2017}. (The references to Wang in the first paragraph are to Liquan Wang.) To state Liu's expansion formulas, we need the notation of $q$-rising factorials from Gasper and Rahman \cite{GR90}. 
The $q$-rising factorial is defined as $\pqrfac{A}{0}{q} :=1$, and when $k$ is a positive integer, 
\begin{equation*}
\pqrfac{A}{k}{q} := (1-A)(1-Aq)\cdots (1-Aq^{k-1}).
\end{equation*}
When $k=\infty$, we require $|q|<1$ for absolute convergence of the infinite product. 
Observe that for  $|q|<1$, 
\begin{equation*}
\pqrfac{A}{k}{q} = \frac{\pqrfac{A}{\infty}{q}}{\pqrfac{Aq^k}{\infty}{q}},
\end{equation*}
an identity that is used to define $q$-rising factorials when $k$ is a complex number. 
In what follows, we drop the `base' $q$ from our displays, and
use the compressed notation
$$\qrfac{a_1, a_2, \dots, a_n}{k}=\qrfac{a_1}{k}\qrfac{a_2}{k}\dots \qrfac{a_n}{k}.$$
We refer the reader to \cite{GR90} for the definition of $_r\phi_s$ series, and terms used to describe such series, such as balanced and very-well-poised, as well as standard summation and transformation formulas in this area.

Liu's main expansion formula~\cite[Theorem~1.1]{Liu2013b} is as follows. Let $f(x)$ be an analytic function near $x=0$; then, assuming suitable convergence conditions, 
\begin{multline}\label{liu-main1}
\frac{\qrfac{\alpha q, \alpha ab/q}{\infty}}{\qrfac{\alpha a, \alpha b}{\infty}}f(\alpha a) 
=\sum_{k=0}^\infty 
\frac{(1-\alpha q^{2k})\qrfac{\alpha, q/a}{k}}
{(1-\alpha)\qrfac{q, \alpha a}{k}}\Big(\frac{a}{q}\Big)^k
\sum_{j=0}^k
\frac{\qrfac{q^{-k}, \alpha q^k}{j}}{\qrfac{q, \alpha b}{j}} q^j f(\alpha q^{j+1}).
\end{multline}
From this, Liu~\cite[Theorem~4.1]{Liu2013a} obtains a general transformation formula: 
\begin{multline}\label{liu-gen1}
\frac{\qrfac{\alpha q, \alpha ab/q}{\infty}}{\qrfac{\alpha a, \alpha b}{\infty}}
\sum_{j=0}^\infty \qrfac{q/a}{j} (\alpha a)^j A_j  \\
=\sum_{k=0}^\infty 
\frac{(1-\alpha q^{2k})\qrfac{\alpha, q/a, q/b}{k}}
{(1-\alpha)\qrfac{q, \alpha a, \alpha b}{k}}\Big(-\frac{\alpha ab}{q}\Big)^k q^{\binom{k}2}
\sum_{j=0}^k
\frac{\qrfac{q^{-k}, \alpha q^k}{j}}{\qrfac{q/b}{j}} \Big(\frac{q^2}{b}\Big)^j A_j.
\end{multline}
In addition to these two, a third key formula of Liu is mentioned below.

Our work is in the context of multiple basic hypergeometric series over root systems. This type of series has been developed systematically (though sporadically) over the last many years; see Schlosser \cite{MS2020} for an encyclopedic survey and a comprehensive list of references. Previously, many of the central results of basic hypergeometric series have been extended.
The goal of this paper is to lift Liu's expansion formulas to basic hypergeometric series over root systems. 
 
In fact, we have exceeded our goal: we have extended Wang and Ma's~\cite{WM2017} more general expansion formulas.
While Liu ~ \cite{Liu2002, Liu2003} 
used techniques of $q$-calculus to derive \eqref{liu-main1}, 
Wang and Ma used $q$-Lagrange inversion, which---as Gessel and Stanton
\cite{GS1983} pointed out---is equivalent to matrix inversion; in this case, to the well-known Bailey transform. Wang and Ma's approach can be extended to multiple basic hypergeometric series over root systems, since the Bailey transform has been extended to this context.  Further, their results are generalizations of Liu's expansion formulas.

 Wang and Ma's 
 transformation formula~\cite[Cor.\ 2.15]{WM2017}  (slightly rewritten) is as follows.
\begin{multline}\label{wang-ma2}
\frac{\qrfac{A,aAy/b, bqy, aq}{\infty}}{\qrfac{Ay,aA/b, bq, aqy}{\infty}}
\sum_{j=0}^\infty \frac{(1-bq^{2j})\qrfac{1/y}{j}}{(1-b)\qrfac{bqy}{j}} y^j \beta_j  \\
=\sum_{k=0}^\infty 
\frac{(1-a q^{2k})\qrfac{1/y, a, bq/A, a/b}{k}}
{(1-a)\qrfac{q, aqy, bq,aA/b}{k}}\big(Ay\big)^k \\
\times
\sum_{j=0}^k
\frac{(1-bq^{2j})\qrfac{q^{-k}, aq^k, A}{j}}{(1-b)\qrfac{bq/A, bq^{k+1}, bq^{1-k}/a}{j}}
 \Big(\frac{bq}{aA}\Big)^j \beta_j.
\end{multline}
The following expansion formula, implicit in  the proof of \cite[Th.\ 1.7]{WM2017}, is in the spirit of \eqref{liu-main1}. 
Suppose both sides of the following are analytic in $y$ in a disk around the origin.
Then
\begin{multline}\label{wang-ma1}
K(y)
\sum_{j=0}^\infty \frac{(1-bq^{2j})\qrfac{1/y}{j}}{(1-b)\qrfac{bqy}{j}} y^j \beta_j 
=\sum_{k=0}^\infty 
\frac{(1-a q^{2k})\qrfac{1/y, a}{k}}
{(1-a)\qrfac{q, aqy}{k}} y^k \\
\times
\sum_{j=0}^k
\frac{\qrfac{q^{-k}, aq^k, bq}{j}}{\qrfac{b}{2j}}
(-1)^j q^{\binom{j+1}2}\beta_j\\
\times 
\sum_{m=0}^{k-j}
\frac{\qrfac{q^{-k+j}, aq^{j+k},bq^{j+1}}{m}}{\qrfac{q, bq^{1+2j}}{m}\qrfac{aq}{j+m}}
q^{m}K\big(q^{j+m}\big).
\end{multline}

Wang and Ma's results contain the aforementioned results of Liu. 
To obtain \eqref{liu-main1}, take 
$$K(y)=\frac{\qrfac{\alpha q,\alpha by}{\infty}}{\qrfac{\alpha qy,\alpha b}{\infty}} f(\alpha qy);
a\mapsto \alpha; \text{ and, } \beta_j = \delta_{j,0}=
\begin{cases} 1, & \text{if } j=0;\\
0, & \text{otherwise},
\end{cases}
$$
in \eqref{wang-ma1}, and then replace $y$ by $a/q$ in the resulting identity. 
Liu's second result
\eqref{liu-gen1} follows from \eqref{wang-ma2}: take
$a\mapsto \alpha, A\mapsto Ab, y\mapsto a/q, b_j\mapsto (\alpha q)^j A_j;$
set $b=0$, and finally replace $A$ by $b$. 

In this paper, we provide several extensions of \eqref{wang-ma2} and \eqref{wang-ma1} to multiple series over root systems. 
The series are all of the form
$$
\sum\limits_{\substack{{k_r\geq 0 } \\  
{r =1,2,\dots, n}}} S(\k)
$$
where $\k=(\kvec )$, and  $k_1, k_2, \dots, k_n$ are non-negative integers. The positive integer $n$ is called the {\em dimension} of the sum. When $n=1$, we refer to the corresponding identity as {\em one-variable}, even though it may have many variables/parameters.  
This type of series is recognized by the presence of the so-called  \lq\lq Vandermonde factor" of type $A$, namely
$$\vandermonde{x}{k}{n}.$$
Some of the series have additional factors that give type $C$ or $D$ Vandermonde factors.
%

To explain why there are several multiple series extensions we need to understand Wang and Ma's approach. In what follows, we have made several changes in their exposition to bring it closer 
to Gessel and Stanton's \cite{GS1983}  explanation.

Let $\F=(F_{km}(a))$ and $\G=(G_{km}(a))$ be  the pair of inverse, infinite, lower triangular matrices, with entries given by 
\begin{subequations}
\begin{align}
F_{km}(a) &:=\label{BaileyF}
\frac{(1-aq^{2m})\qrfac{q^{-k}}{m}}
{\qrfac{aq^{k+1}}{m}} q^{km};\\
\intertext{and}
G_{km}(a) &:=\label{BaileyG}
\frac{\qrfac{q^{-k}}{m}\qrfac{aq^{m}}{k}}
{(1-aq^m)\qrfac{q}{m}\qrfac{q}{k}} q^{m}.
\end{align}
\end{subequations}
The matrix inversion is equivalent to the Bailey matrix inversion (take $p=q$ in
\cite[Th.\ 3.7]{GS1983} 
or $n=1$ in \cite[Th.\ 3.7]{Milne1997}). We obtain \eqref{BaileyF} and \eqref{BaileyG} on multiplication by suitable diagonal matrices. 

Consider 
\begin{equation}\label{fk-1}
f_k(y; a) := \frac{(1-aq^{2k})\qrfac{1/y}{k}}
{\qrfac{aqy}{k}} y^{k} .
\end{equation}
The motivation for $f_k(y;a)$ is that $f_k(q^m;a)=F_{mk}(a)$.

Our goal is to write an expression $A(y)$ (to be prescribed shortly) in the form
\begin{equation}\label{Ay}
A(y)= \sum_{k=0}^\infty f_k(y;a) \alpha_k.\tag{$*$}
\end{equation}
The approach is as follows. Since
$$A(q^N) = \sum_{k=0}^N F_{Nk}(a) \alpha_k,$$
we have, by matrix inversion,
$$\alpha_k = \sum_{m=0}^k G_{km}(a) A(q^m);$$
using this we obtain an expression for $A(q^N)$. The hope is that this expression implies (by analytic continuation) an expression of the form \eqref{Ay} for $A(y)$. We take $A(y)$ to be of the form
\begin{equation}\label{Ayb}
A(y) = K(y)\sum_{j=0}^\infty h_j(y) \beta_j.\tag{$**$}
\end{equation}
So we obtain:
\begin{align}
A(q^N) &= \sum_{k=0}^N F_{Nk}(a) \sum_{m=0}^k G_{km}(a) A(q^m)\cr
&=\sum_{k=0}^N F_{Nk}(a) \sum_{m=0}^k G_{km}(a) K(q^m) \sum_{j=0}^m h_j(q^m) \beta_j\cr
& = \sum_{k=0}^N F_{Nk}(a) \sum_{j=0}^k \beta_j \sum_{m=0}^{k-j} h_j(q^{j+m}) G_{k,j+m}(a) K(q^{j+m}),
\label{wang-ma-laststep}
\end{align}
by interchanging the two inner sums.
One can take any function $h_j(y)$ as long as \eqref{Ayb} is analytic in a disk around the origin. The choice Wang and Ma made is (essentially) $h_j(y)=f_j(y,b)$. With this choice, we find that
$$
A(q^N)=  \sum_{k=0}^N F_{Nk}(a) \sum_{j=0}^k \beta_j \sum_{m=0}^{k-j} F_{j+m,j}(b) G_{k,j+m}(a) K(q^{j+m}) .
$$
After plugging in and simplifying, the resulting identity is \eqref{wang-ma1}, with $y$ replaced by $q^N$. Under the assumption that $K(y)$ and $\beta_j$ are chosen so that both sides of \eqref{wang-ma1} are analytic in $y$
in a disk around the origin, we obtain \eqref{wang-ma1} by analytic continuation. 

To obtain \eqref{wang-ma2},  we choose $K(y)$ in \eqref{wang-ma1} so that the inner sum is summed using the $q$-Pfaff-Saalch{\"u}tz summation \cite[Eq.\ (1.7.2)]{GR90}. This is the sum of a terminating, balanced $_3\phi_2$ sum. 

There are two important observations in the exposition above that we require in our work ahead.
\begin{enumerate}
\item The choice of $K(y)$  is such that  the $q$-Pfaff-Saalch{\"u}tz theorem applies to the inner sum. So the statement of \eqref{wang-ma1} can be generalized but this case is likely most useful due to the fact that it implies \eqref{wang-ma2} and contains \eqref{liu-gen1}.
\item When $\beta_j=\delta_{j,0}$, then \eqref{wang-ma2} reduces to Rogers' $_6\phi_5$ summation~\cite[Eq.\ (2.7.1)]{GR90}. In case we don't apply analytic continuation, it reduces to the terminating very-well-poised $_6\phi_5$ sum \cite[Eq.\ (2.4.2)]{GR90}. 
\end{enumerate}

To return to the description of results in this paper, we recall that there are extensions of the Bailey transform over the root systems $A_n$ and $C_n$ given in, respectively, 
Milne~\cite{Milne1997} and Milne and Lilly~\cite{ML1995}; two non-terminating $_6\phi_5$ summations due to Milne~\cite{Milne1987, Milne1988} and Gustafson~\cite{RG1989}, and more terminating $_6\phi_5$ summations and several balanced  $_3\phi_2$ summations due to Milne~\cite{Milne1997} and the first author \cite{GB1999a}. All of these combine to motivate the several multiple series generalizations presented in this paper. We obtain
three extensions of \eqref{wang-ma1}; and, as special cases, four extensions of \eqref{wang-ma2}. 

This paper is organized as follows. In \S\ref{sec:trans1} we give $A_n$ extensions of \eqref{wang-ma2} and \eqref{wang-ma1}; in \S\ref{sec:cn-dn}, extensions over $C_n$ and $D_n$. In \S\ref{sec:app1}, we show how one can apply these results to find expansions of arbitrary infinite products. This is based on a theorem of Liu extending the non-terminating $_6\phi_5$ summation. In \S\ref{sec:liu3} we provide one example of how to modify our calculations to extend a third formula of Liu, namely, Theorem 9.2 of Liu~\cite{Liu2013c}:
\begin{multline}\label{liu3}
\frac{\qrfac{aq,aAB/q}{N}}{\qrfac{aA, aB}{N}}
\sum_{j=0}^N \frac{\qrfac{q^{-N}, q/A, q/B}{j}}{\qrfac{q^{2-N}/aAB}{j}} q^j A_j  \\
=\sum_{k=0}^N 
\frac{(1-a q^{2k})\qrfac{q^{-N}, a, q/A, q/B}{k}}
{(1-a)\qrfac{q, aq^{N+1}, aA,aB}{k}}\big(aABq^{N-1}\big)^k \\
\times
\sum_{j=0}^k
\big(q^{-k}, aq^k\big)_{j} q^j
 A_j.
\end{multline}
To obtain \eqref{liu3}, we require a slightly different $h_k(y)$. With a suitable choice of $K(y)$, we again use the $q$-Pfaff-Saalch{\"u}tz summation to sum the inner sum. See \S\ref{sec:liu3}. This provides an alternate, simpler, proof of Liu's formula. 

We remark that Liu's three identities mentioned above have not---until now---been brought into the framework of the Bailey transform 
and lemma. Wang and Ma showed some of Liu's results 
using $q$-Lagrange inversion, and that coupled with Gessel and Stanton's well-known observation 
is what prompted us to  examine Liu's work from this point of view. 

\section{$A_n$ expansion and transformation formulas}\label{sec:trans1}
In this section we give our first set of results pertaining to $A_n$ multiple series. We first provide an alternative expression for the 
$A_n$ Bailey transform and use that to motivate the definition of a function extending $f_k(y;a)$. 

In what follows we assume the parameters $a$, $b$, $c$, $A$, $x_k$, for  $k=1, 2, \dots, n$, and  $q$ are all complex numbers with $|q|<1$. We assume in all results below that the parameters are chosen so that the denominators are never $0$. 
We use $\N=(N_1, N_2, \dots, N_n)$, where $N_r$ are non-negative integers; similarly 
$\k$, $\j$, and $\m$ are multi-indices of non-negative integers. We use the notation $\sumk:= \sumvec{k}{n}$ and 
$\j+\m = (j_1+m_1, \dots, j_n+m_n)$ for component-wise addition. We will use the variables
$\y=(y_1, y_2, \dots, y_n)$. We replace  $\y$ by $q^{\N}$ in any formula to indicate  $y_r$ has been replaced by $q^{N_r}$ for $r=1, 2, \dots, n$.   
The multi-indices can be ordered lexicographically; this allows us to 
define matrix multiplication of matrices with rows and columns indexed by multi-indices.

\begin{Proposition}[An alternative expression for the $A_n$ Bailey transform]\label{prop:bailey-inversion}
Let \!
$\F= \big(F_{\k\m}(a)\big)$ and $\G = \big(G_{\k\m}(a)\big)$ be matrices with entries indexed by $(\k , \m)$ defined as 
\begin{subequations}
\begin{align}
F_{\k\m}(a) &:= \label{An-Bailey-F}
\vandermonde{x}{m}{n}  \;
\powerq{m}{n} q^{\sumk\summ} \cr
&\hspace{1cm}
\sqprod{n}\qrfac{q^{-k_s}\xover{x}}{m_r} 
\smallprod n  \frac{1-ax_rq^{m_r+\summ}}{\qrfac{ax_r q^{k_r+1}}{\summ}};
\intertext{and}
G_{\k\m}(a) &:= \label{An-Bailey-G}
\vandermonde{x}{m}{n}  \;
 q^{\sum_{r=1}^n r m_r} \cr
&\hspace{1cm}
\sqprod{n}\frac{\qrfac{q^{-k_s}\xover{x}}{m_r} }{\qrfac{q\xover{x}}{k_r}\qrfac{q\xover{x}}{m_r}}
\smallprod n  \frac{\qrfac{ax_r q^{m_r}}{\sumk}}{1-ax_rq^{m_r}}.
\end{align}
\end{subequations}
Then $\F$ and $\G$ are inverses of each other.
\end{Proposition}
\begin{Remark} The entries $F_{\k\m}(a)$ and $G_{\k\m}(a)$ are $0$ unless $m_r\le k_r$, for $r=1,2,\dots, n.$ This follows from
$$\qrfac{q^{-k}}{m} = 0 \text{ whenever } m>k.$$
Thus, under the lexicographic ordering of the multi-indices, the matrices are lower-triangular. 
\end{Remark}
\begin{proof} The proposition is an equivalent formulation of Milne~\cite[Th.\ 3.3]{Milne1997}. It is obtained by multiplying the matrices in Milne's formulation by suitable diagonal matrices. The details are as follows. 

Let $\M$ and $\M^*$ be the matrices with entries given in Milne~\cite[eq.\ (3.2a) and (3.2b)]{Milne1997}. We 
take the $a\mapsto ax_n$ case of the entries, and rewrite them using the useful computation \cite[Lemma~4.3]{Milne1997}:
\begin{multline}\label{magiclemma2}
\sqprod n 
\frac{1}{\qrfac{q^{1+m_r-m_s}\xover x }{k_r-m_r}} =
\vandermonde{x}{m}{n} \cr
\sqprod n \frac{\qrfac{q^{-k_s}\xover{x}}{m_r} }{\qrfac{q\xover{x}}{k_r} } 
(-1)^{\summ} q^{\sumk\summ-\binom{\summ}2 }q^{\sum\limits_{r=1}^n (r-1)m_r}
.
\end{multline} 
Denote the resulting matrices by $\M$ and $\M^*$ too. Consider the
matrices
$
\F = \D \M \D^\prime$  and 
$\G = {\D^\prime}^{-1} \M^* \D^{-1} ,$
where the diagonal entries of the diagonal matrices $\D$ and $\D^\prime$ are given by
\begin{align*}
D_{\k\k} &= \smallprod n \qrfac{ax_rq}{k_r} \sqprod n \qrfac{q\xover{x} }{k_r}\cr
\intertext{and}
D^\prime_{\m\m} &= (-1)^{\summ} q^{\binom{\summ}{2}}\smallprod n \big(1-ax_r q^{m_r+\summ}\big).
\end{align*}
Note that multiplying the matrix $\M$ by a diagonal matrix on the left multiplies each entry of $\M$ by a factor depending only on the row index 
$\k$, and multiplying by a diagonal matrix on the right multiplies each entry by a factor depending on the column index $\m$. In this manner we obtain the entries of $\F$ and $\G$ given in \eqref{An-Bailey-F}
and \eqref{An-Bailey-G}, respectively. 

Since $\M$ and $\M^*$ are inverses (by \cite[Th.\ 3.3]{Milne1997}), so are $\F$ and $\G$. 
\end{proof}

Let $f_{\k}(\y;a)$ be 
defined as
\begin{multline}\label{fk}
f_{\k}(\y;a):= 
\vandermonde{x}{k}{n} 
\sqprod{n}\qrfac{\xover{x}y_s}{k_r} \\
\smallprod n \frac{1-ax_rq^{k_r+\sumk}}{\qrfac{aqx_ry_r}{\sumk}}
(y_1y_2\dots y_n)^\sumk 
\powerq{k}{n}. 
\end{multline}
The motivation is that  $f_{\k}(q^{\m})$ reduces to the entries  of the $A_n$ Bailey transform in the form $F_{\m\k}(a)$ given in \eqref{An-Bailey-F}; further, it is analytic in $\y=(\yvec)$.

\begin{Theorem}[A multiple series extension of \eqref{wang-ma1}]\label{th:an-trans1} Let $K(\y)$ and $\beta(\j)$ be chosen so that the series on both sides are analytic in $y_1, y_2, \dots, y_n$ and converge absolutely for each $y_r$ in a disk  around the origin; or, $y_r=q^{N_r}$, for $r=1, 2, \dots, n$, and so the series terminate.  Then
\begin{align}\label{an-trans1}
K(\y)
\multisum{j}{n} &
\vandermonde{x}{j}{n} 
\sqprod{n}\qrfac{\xover{x}y_s}{j_r} \cr
& 
\smallprod n \frac{1-bx_rq^{j_r+\sumj}}{(1-bx_r)\qrfac{bqx_ry_r}{\sumj}} 
 (y_1y_2\dots y_n)^\sumj 
 \powerq{j}{n}
\beta(\j) \cr
=
\multisum{k}{n} &
\vandermonde{x}{k}{n} 
\sqprod{n}\frac{\qrfac{\xover{x}y_s}{k_r} }{\qrfac{q\xover{x}}{k_r} }\cr
&\smallprod n \frac{1-ax_rq^{k_r+\sumk}}{1-ax_r} 
\smallprod n \frac{\qrfac{ax_r}{\sumk}}{\qrfac{aqx_ry_r}{\sumk}}
(y_1y_2\dots y_n)^\sumk \powerq{k}{n} \cr
&\times 
\multsum{j}{k}{r}
\vandermonde{x}{j}{n} 
\sqprod{n}\qrfac{q^{-k_s}\xover{x}}{j_r} \cr
&\hspace{1.75cm}
\smallprod n \frac{\qrfac{ax_rq^{\sumk}, bqx_r}{j_r}}
{\qrfac{bx_r}{j_r+\sumj}}
(-1)^{\sumj} q^{\suml_{r=1}^{n} rj_r + \binom{\sumj}{2}} \beta(\j)\cr
&\times
\underset{r=1, \dots, n}{\sum\limits_{0\leq m_r\leq k_r-j_r}}
\triprod n \frac{1-q^{j_r-j_s+m_r-m_s} \xover {x} }{1-q^{j_r-j_s}  \xover{x}}\;
q^{\suml_{r=1}^n rm_r} K(q^{\j+\m})
\cr
&\hspace{2.25cm}
\sqprod{n}\frac{\qrfac{q^{j_r-k_s}\xover{x}}{m_r} }{\qrfac{q^{1+j_r-j_s}\xover{x}}{m_r} }
\smallprod n \frac{\qrfac{ax_rq^{j_r+\sumk}, bx_rq^{1+j_r}}{m_r}}
{\qrfac{aqx_r}{j_r+m_r}\qrfac{bx_rq^{1+j_r+\sumj}}{m_r}}
.\cr
\end{align}
\end{Theorem}
\begin{proof}
We will explain the proof symbolically, keeping the explicit details to a minimum. The details involve a large number of routine simplifications, some of which are mentioned below.  

We begin with an expression $A(\y)$ which we want to write in the form
$$A(\y)=\multisum{k}{n} f_{\k}(\y;a) \alpha(\k)$$
with parameter $a$. 
When $y_r=q^{N_r}$, for $r=1, 2, \dots, n$,
$$A(q^{\N}) = \multsum{k}{N}{r} F_{\N\k}(a) \alpha(\k). $$
The matrix inversion in Proposition~\ref{prop:bailey-inversion} implies the inverse relation:
$$\alpha(\k) = \multsum{m}{k}{r} G_{\k\m}(a) A(q^{\m}), $$
so
$$A(q^{\N})=\multsum{k}{N}{r} F_{\N\k}(a) \multsum{m}{k}{r} G_{\k\m} A(q^{\m})$$
We now assume that 
\begin{equation*}
A(\y):= K(y)\multisum{j}{n} f_{\j}(\y;b) \beta(\j)
\end{equation*}
containing a parameter $b$.
Using this form in the above, we obtain:
$$A(q^{\N})=\multsum{k}{N}{r} F_{\N\k}(a) \multsum{m}{k}{r} G_{\k\m} (a)
K(q^{\m}) \multsum{j}{m}{r} f_{\j}(q^{\m};b) \beta(\j). 
$$
Next interchange the inner sums using 
$$\multsum{m}{k}{r} \; \multsum{j}{m}{r} S(\m,\j) =
\multsum{j}{k}{r} \;
\underset{r=1, \dots, n}{\sum\limits_{0\leq m_r\leq k_r-j_r}} S(\j+\m,\j).$$
In the resulting expression, we divide both sides by
$$\smallprod n (1-bx_r)$$
to obtain
\begin{align}\label{trans1-terminating}
K(q^{\N})
\multsum{j}{N}{r} &
\vandermonde{x}{j}{n} 
\sqprod{n}\qrfac{q^{-N_s}\xover{x}}{j_r} \cr
& 
\smallprod n \frac{1-bx_rq^{j_r+\sumj}}{(1-bx_r)\qrfac{bx_rq^{1+N_r}}{\sumj}} 
 q^{\sumj\sumN} 
 \powerq{j}{n}
\beta(\j) \cr
=
\multsum{k}{N}{r} &
\vandermonde{x}{k}{n} 
\sqprod{n}\frac{\qrfac{q^{-N_s}\xover{x}}{k_r} }{\qrfac{q\xover{x}}{k_r} }\cr
&\smallprod n \frac{1-ax_rq^{k_r+\sumk}}{1-ax_r} 
\smallprod n \frac{\qrfac{ax_r}{\sumk}}{\qrfac{ax_rq^{1+N_r}}{\sumk}}
q^{\sumk\sumN} \powerq{k}{n} \cr
&\times 
\multsum{j}{k}{r}
\vandermonde{x}{j}{n} 
\sqprod{n}\qrfac{q^{-k_s}\xover{x}}{j_r} \cr
&\hspace{1.75cm}
\smallprod n \frac{\qrfac{ax_rq^{\sumk}, bqx_r}{j_r}}
{\qrfac{bx_r}{j_r+\sumj}}
(-1)^{\sumj} q^{\suml_{r=1}^{n} rj_r + \binom{\sumj}{2}} \beta(\j)\cr
&\times
\underset{r=1, \dots, n}{\sum\limits_{0\leq m_r\leq k_r-j_r}}
\triprod n \frac{1-q^{j_r-j_s+m_r-m_s} \xover {x} }{1-q^{j_r-j_s}  \xover{x}}\;
q^{\suml_{r=1}^n rm_r} K(q^{\j+\m})
\cr
&\hspace{2.25cm}
\sqprod{n}\frac{\qrfac{q^{j_r-k_s}\xover{x}}{m_r} }{\qrfac{q^{1+j_r-j_s}\xover{x}}{m_r} }
\smallprod n \frac{\qrfac{ax_rq^{j_r+\sumk}, bx_rq^{1+j_r}}{m_r}}
{\qrfac{aqx_r}{j_r+m_r}\qrfac{bx_rq^{1+j_r+\sumj}}{m_r}}
.\cr
\end{align}
This is \eqref{an-trans1} with $y_r=q^{N_r}$, for $r=1, 2, \dots, n$, and proves the result when the sums are terminating.
Next assume $A(\y)$ is  analytic in $y_r$  in a domain of the form
$|y_r|\le \epsilon_r$ containing $0$ (for $r=1, 2, \dots, n$). Since \eqref{trans1-terminating} shows  \eqref{an-trans1} is true for
$y_r=q^{N_r}$ for $r=1, 2, \dots, n$ in turn for $N_r=0, 1, 2, \dots$, and $|q|<1$, it is true in the domain of analyticity by the Identity Theorem. 

%
%
%

We have omitted a large number of elementary computations. For example, we require
\begin{align}\label{elem1}
&\sqprod{n}\frac{\qrfac{q^{-k_s}\xover{x}}{j_r+m_r}
 \qrfac{q^{-m_s-j_s}\xover{x}}{j_r} }
 {\qrfac{q\xover{x}}{j_r+m_r} }
= \cr
&\hspace{1cm}
\sqprod{n}\frac{\qrfac{q^{-k_s}\xover{x}}{j_r} 
\qrfac{q^{j_r-k_s}\xover{x}}{m_r} }
{
\qrfac{q^{1+j_r-j_s}\xover{x}}{m_r} }\cr
&\hspace{1.5cm}
\triprod n \frac{1-\xover x}{1-q^{j_r-j_s}\xover x}
(-1)^\sumj q^{-\suml_{r=1}^n (r-1)j_r-\sumj\summ-\sumj^2+\binom{\sumj}2}.
\end{align}
The calculations require elementary identities from \cite[Appendix~1]{GR90}, and the helpful simplification \cite[Lemma~3.12]{Milne1997}:
\begin{multline*}
\sqprod n 
{\qrfac{q\xover x }{j_r-j_s}} =
\vandermonde{x}{j}{n}  
(-1)^{(n-1)\sumj} \cr
\times q^{\suml_{r=1}^n (r-1)j_r+
 n\suml_{r=1}^n\binom{j_r}{2} -\binom{\sumj}2 } \smallprod n x_r^{nj_r-\sumj}
.
\end{multline*}

\end{proof}

The next idea is to choose $K(\y)$ in such a way that the inner sum is summable by a terminating, $A_n$ balanced
$_3\phi_2$ summation result. We will use a result of Milne to find the following.

\begin{Theorem}[An $A_n$ extension of \eqref{wang-ma2}]\label{th:an-result5a} 
Let $a$, $A$, $b$, $x_1,x_2,\dots, x_n$ and $\beta(\j)$ be such that the 
series on both sides are analytic in $y_1, y_2, \dots, y_n$ and converge absolutely for each $y_r$ in a disk  around the origin; or, $y_r=q^{N_r}$, for $r=1, 2, \dots, n$, and so the series terminate.
 Then we have
\begin{align}\label{an-result5a}
\frac{\qrfac{A, aAy_1\dots y_n/b}{\infty}}{\qrfac{aA/b, Ay_1\dots y_n}{\infty}} &
\smallprod n \frac{\qrfac{bqx_ry_r, aqx_r}{\infty}}{\qrfac{aqx_ry_r, bqx_r}{\infty}} \cr
\multisum{j}{n} &
\vandermonde{x}{j}{n}  
\sqprod{n}\qrfac{\xover{x}y_s}{j_r} \cr
&\smallprod n \frac{1-bx_rq^{j_r+\sumj}}{(1-bx_r)\qrfac{bqx_ry_r}{\sumj}}
(y_1y_2\dots y_n)^\sumj\powerq{j}{n} \beta(\j) \cr
=
\multisum{k}{n} &
\vandermonde{x}{k}{n} 
\sqprod{n}\frac{\qrfac{\xover{x}y_s}{k_r} }{\qrfac{q\xover{x}}{k_r} }\cr
&\smallprod n \frac{1-ax_rq^{k_r+\sumk}}{1-ax_r} 
\smallprod n \frac{\qrfac{ax_r}{\sumk} \qrfac{bqx_r/A}{k_r}}{\qrfac{aqx_ry_r}{\sumk}\qrfac{bqx_r}{k_r}} \cr
& \frac{\qrfac{a/b}{\sumk}}{\qrfac{aA/b}{\sumk}}
(Ay_1y_2\dots y_n)^\sumk \powerq{k}{n} \cr
&\times 
\multsum{j}{k}{r}
\vandermonde{x}{j}{n} 
\sqprod{n}\qrfac{q^{-k_s}\xover{x}}{j_r} \cr
&\hspace{2cm}
\smallprod n \frac{(1-bx_rq^{j_r+\sumj})\qrfac{ax_rq^{\sumk}}{j_r}}
{
(1-bx_r)\qrfac{bqx_r/A}{j_r} \qrfac{bx_rq^{k_r+1}}{\sumj}}\cr
&\hspace{2cm}
\frac{\qrfac{A}{\sumj}}{\qrfac{bq^{1-\sumk}/a}{\sumj}}
 q^{\suml_{r=1}^{n} (r-1)j_r } \Big(\frac{bq}{aA}\Big)^{\sumj}\beta(\j)
.
\end{align}
\end{Theorem}
\begin{Remarks} \ 
\begin{enumerate}
\item An $A_n$ non-terminating, very-well-poised $_6\phi_5$ summation is contained in 
 \eqref{an-result5a}. Take
$$\beta_{j}=\prod_{r=1}^n \delta_{j_r,0}, b\mapsto a/b, y_r\mapsto 1/c_r, \text{ and  } A\mapsto aq/bd$$
to obtain an equivalent form of 
a theorem  of Milne~\cite{Milne1987} (see~\cite[Th.\ A.3]{MN2012}).
\item For convergence we use the multiple power series test. For a discussion on convergence of multiple series and examples, see Milne~\cite{Milne1997} and Schlosser \cite{MS2005}. As an example, let $\epsilon_k$ be the $n$-tuple with all entries $0$ except the $k$th which is $1$. Suppose $\beta(\j)$ is such that $|\beta(\j + \epsilon_k)/\beta(\j)|\le1$ for all $k$.  Then the multiple power series test gives the following convergence conditions:  $|y_1y_2\dots y_n|<1$, $|Ay_1y_2\dots y_n|<1$, and, $|bq/aA|<1$. 
\item In the proof below, we need to choose $K(\y)$ appropriately so that the left hand side of 
\eqref{th:an-trans1} is analytic in $y_r$ around $0$ for $r=1, 2, \dots, n$. For this, we keep in mind that products of the form
$\qrfac{A}{\infty}/\qrfac{Ay}{\infty}$ and $\qrfac{Ay}{\infty}/\qrfac{A}{\infty}$  are analytic in $y$ 
(provided $|Ay|<1$) by appealing to the $q$-binomial theorems given in equations (1.3.15) and 
(1.3.16) of Gasper and Rahman \cite{GR90}.
\end{enumerate}
\end{Remarks}
\begin{proof} We choose $K(\y)$ in Theorem~\ref{th:an-trans1} so that the inner sum is summable using an $A_n$ balanced $_3\phi_2$ summation formula due to Milne~\cite[Theorem 4.1]{Milne1997}:
\begin{align}\label{an3p2-1}
\frac{\qrfac{c/a}{\sumN}}{\qrfac{c/ab}{\sumN}} &
\smallprod n \frac{\qrfac{cx_r/b}{N_r}}{\qrfac{cx_r}{N_r}} \cr
&=
\multsum{k}{N}{r} 
\vandermonde{x}{k}{n} 
\sqprod{n} \frac{\qrfac{q^{-N_s}x_r/x_s}{k_r}}{\qrfac{qx_r/x_s}{k_r}} \cr
&\hspace{1.5 cm}\times 
\smallprod n \frac{\qrfac{ax_r}{k_r}}{\qrfac{cx_r}{k_r}} 
 \frac{\qrfac{b}{\sumk}}{\qrfac{abq^{1-\sumN}/c}{\sumk}}
q^{\suml_{r=1}^{n} rk_r } .
\end{align}

We consider the transformation formula \eqref{an-trans1} with 
$$K(\y)=
 \frac{\qrfac{A, aAy_1\dots y_n/b}{\infty}}{\qrfac{aA/b, Ay_1\dots y_n}{\infty}}
 \smallprod n \frac{\qrfac{aqx_r, bqx_ry_r}{\infty}}{\qrfac{aqx_ry_r, bqx_r}{\infty}}.
 $$
 After some elementary calculations using \cite[Appendix~1]{GR90}, we apply 
 the
 $$x_r\mapsto x_rq^{j_r}, N_r=k_r-j_r, a\mapsto aq^{\sumK}, b\mapsto Aq^{\sumj}
 , \text{ and } c\mapsto bq^{\sumj+1}$$
 case of
 \eqref{an3p2-1}
 to sum the inner-most sum on the right hand side of \eqref{an-trans1}. The resulting identity can be written as \eqref{an-result5a} after some further simplifications.
\end{proof}

\begin{remark}\label{note-convergence}
Before proceeding to the next result, we remark on the applicability of Theorem~\ref{th:an-trans1}. Note that the factors of the form $\qrfac{q^{-k_s}x_r/x_s}{j_r}$ in the middle sum become very large, and will make the series on the right hand side diverge. There are similar factors in the innermost sum which need to be addressed by a judicious choice of $K(\y)$ and $\beta(\j)$.  
However, in \eqref{an-result5a}, this factor does not make the series diverge because of the compensating factor of the form 
$\qrfac{Aq^{-\sumk}}{\sumj}$ in the denominator. Thus, it may be easier to apply Theorem~\ref{th:an-result5a}. This remark applies to many results of this paper. 
\end{remark}

Next, we apply the same idea and sum the inner sum using another $A_n$ balanced $_3\phi_2$ summation of Milne. In this case, calculations analogous to those made in the proof of Theorem~\ref{th:an-result5a} yield a non-terminating series on the right hand side which does not converge. So we obtain a result for terminating series. 

\begin{Theorem}[A terminating extension of \eqref{wang-ma2} over $A_n$]\label{th:an-result5b}
\begin{align}\label{an-result5b}
 &\hspace{-2cm}
 \smallprod n  \frac{\qrfac{Ax_r,  aqx_r}{N_r}}
 {\qrfac{ aAx_r/b, bqx_r}{N_r}} 
 \cr
 \multsum{j}{N}{r} &
\vandermonde{x}{j}{n}  
\sqprod{n}\qrfac{q^{-N_s}\xover{x}}{j_r} \cr
&\smallprod n \frac{(1-bx_rq^{j_r+\sumj})}{(1-bx_r)\qrfac{bx_rq^{1+N_r}}{\sumj}}
q^{\sumj\sumN}\powerq{j}{n} \beta(\j) \cr
=
\multsum{k}{N}{r} &
\vandermonde{x}{k}{n} 
\sqprod{n}\frac{\qrfac{q^{-N_s}\xover{x}}{k_r} }{\qrfac{q\xover{x}}{k_r} }\cr
&\smallprod n \frac{1-ax_rq^{k_r+\sumk}}{1-ax_r} 
\smallprod n \frac{\qrfac{ax_r}{\sumk}}{\qrfac{ax_rq^{1+N_r}}{\sumk}\qrfac{bqx_r, aAx_r/b}{k_r} }  \cr
& \hspace{1cm}\qrfac{a/b, bq/A }{\sumk}
(Aq^{\sumN})^\sumk  q^{\suml_{r=1}^{n} (r-1)k_r-\suml_{r<s}k_r k_s}
  \smallprod n x_r^{k_r} \cr
&\times 
\multsum{j}{k}{r}
\vandermonde{x}{j}{n} 
\sqprod{n}\qrfac{q^{-k_s}\xover{x}}{j_r} \cr
&\hspace{2cm}
\smallprod n \frac{(1-bx_rq^{j_r+\sumj}) \qrfac{ax_rq^{\sumk}, Ax_r}{j_r}} 
 {(1-bx_r)\qrfac{bx_rq^{k_r+1}}{\sumj}}  \frac{1}{\qrfac{bq/A, bq^{1- \sumk} /a}{\sumj}} \cr
&\hspace{2.5cm}
 \Big(\frac{bq}{aA}\Big)^{\sumj} q^{\suml_{r=1}^{n} (r-1)j_r+\suml_{r<s}j_r j_s}   \smallprod n x_r^{-j_r}
 \beta(\j)
.
\end{align}
\end{Theorem}
\begin{Remark} Theorem~\ref{th:an-result5b} contains an $A_n$ terminating very-well-poised $_6\phi_5$ summation due to Milne~\cite[Th.\ 2.11]{Milne1997}. In  \eqref{an-result5b}, take
$$\beta_{j}=\prod_{r=1}^n \delta_{j_r,0}, b\mapsto a/b, \text{ and  } A\mapsto aq/bc$$ to obtain 
an equivalent form of Milne's result. 
\end{Remark}
\begin{proof} We choose $K(\y)$ in the terminating version of Theorem~\ref{th:an-trans1} so that the inner sum is summable using an $A_n$  terminating balanced $_3 \phi _2$ summation theorem given by Milne~\cite[Th.~4.18]{Milne1997}:
\begin{align}\label{an3p2-2}
&\hspace{-1.5cm}
\frac{  \qrfac{c/a, c/b}{\sumN}}
{ \smallprod n { \qrfac{cx_r , c q^{ \sumN - N_r }/{ab x_r }}{N_r}}}
  \cr
= &
\multsum{m}{N}{r} 
\vandermonde{x}{m}{n} 
\sqprod{n} \frac{\qrfac{q^{-N_s}x_r/x_s}{m_r}}{\qrfac{qx_r/x_s}{m_r}} \cr
& \hspace{1 cm}\times 
\smallprod n \frac{\qrfac{ax_r, bx_r}{m_r}}{\qrfac{cx_r, ab {x_r}q^{1-\sumN}/c}{m_r}} \;
q^{\suml_{r=1}^{n} rm_r } .
\end{align}

We consider the transformation formula \eqref{trans1-terminating} with 
$$K(\y)=
\smallprod n  \frac{\qrfac{Ax_r, aAx_ry_r/b,  aqx_r, bqx_ry_r}{\infty}}
 {\qrfac{ Ax_ry_r, aAx_r/b,,aqx_ry_r, bqx_r}{\infty}} 
 $$
 where $y_r=q^{N_r}$ for $r=1, 2, \dots, n$. 
 After some elementary calculations using \cite[Appendix~1]{GR90}, we apply 
 the
 $$x_r\mapsto x_rq^{j_r}, N_r=k_r-j_r,  a\mapsto aq^{\sumK}, b\mapsto A
 , \text{ and } c\mapsto bq^{\sumj+1}$$
 case of
 \eqref{an3p2-2}
 to sum the inner-most sum obtained from the right hand side of \eqref{trans1-terminating}. The resulting identity can be written as \eqref{an-result5b} after some further simplifications.
\end{proof}

\section{$C_n$ and $D_n$ expansion and transformation formulas}\label{sec:cn-dn}
In this section, we give some variations of the formulas in \S\ref{sec:trans1} where the series involved contain terms that are a mix of  terms usually associated with $A_n$, $C_n$ and $D_n$ series. 

The following is an alternate version of the $C_n$ Bailey transform due to Lilly and Milne. We reuse the symbols from the previous section for this proposition. 
\begin{Proposition}[An alternative expression for the $C_n$ Bailey transform]\label{prop:cn-bailey-inversion}
Let \!
$\F= \big(F_{\k\m}(a)\big)$ and $\G = \big(G_{\k\m}(a)\big)$ be matrices with entries indexed by $(\k , \m)$ defined as 
\begin{subequations}
\begin{align}
F_{\k\m}(a) &:= \label{Cn-Bailey-F1}
\vandermonde{x}{m}{n}  \;
\powerq{m}{n} q^{\sumk\summ} \cr
&\hspace{1cm}
\prod\limits_{1\le r\le s\le n} (1-ax_rx_sq^{m_r+m_s})
\sqprod{n}
\frac{\qrfac{q^{-k_s}\xover{x}}{m_r} }
{\qrfac{ax_r x_sq^{k_s+1}}{m_r}}
\intertext{and}
G_{\k\m}(a) &:= \label{Cn-Bailey-G1}
\triprod{n} \frac{(1-q^{m_r-m_s}\xover x)(1-ax_rx_sq^{m_r+m_s})}
{(1-\xover x)
 }  
 q^{\suml_{r=1}^n r m_r} \cr
&\hspace{1cm}
\sqprod{n}\frac{\qrfac{q^{-k_s}\xover{x}}{m_r} \qrfac{ax_rx_sq^{m_s}}{k_r} 
}
{(1-ax_rx_sq^{m_s}) 
\qrfac{q\xover{x}}{k_r}\qrfac{q\xover{x}}{m_r}
}
\end{align}
\end{subequations}
Then $\F$ and $\G$ are inverses of each other.
\end{Proposition}
\begin{proof} The proposition is an equivalent formulation of Lilly and Milne~\cite[Th.\ 3.19]{LM1993}. It is obtained by multiplying the matrices in Lilly and Milne's formulation by suitable diagonal matrices. The details are as follows. 

Let $\M$ and $\M^*$ denote the matrices obtained by the $x_r\mapsto \sqrt{a}x_r$ case of equations  (3.16) and (3.18) of \cite{LM1993}. This  
inserts an additional (redundant) parameter $a$, and brings them notationally close to the $A_n$ matrices considered earlier. As before, we rewrite them using \eqref{magiclemma2}.
 Let the entries of the diagonal matrices $\D$ and $\D^\prime$ be given by
\begin{align*}
D_{\k\k} &= \sqprod n \qrfac{a qx_rx_s}{k_r} \qrfac{q\xover{x} }{k_r}\cr
\intertext{and}
D^\prime_{\m\m} &= (-1)^{\summ} q^{\binom{\summ}{2}}
\prod\limits_{1\le r\le s\le n}
\Big(1-ax_r x_sq^{m_r+m_s}\Big)
\sqprod n \frac{1}{\qrfac{aqx_rx_s}{m_r+m_s}}.
\end{align*}
Consider the
matrices
$
\F = \D \M \D^\prime$  and 
$\G = {\D^\prime}^{-1} \M^* \D^{-1} .$
The entries of these matrices are given in \eqref{Cn-Bailey-F1}
and \eqref{Cn-Bailey-G1}, respectively. 
\end{proof}

Recall the expression \eqref{fk} for $f_{\k}(\y;a)$ motivated by the $A_n$ Bailey matrix.  We define an analogous expression for $C_n$ as follows.
\begin{multline}\label{gj}
g_{\j}(\y;a):= 
\triprod n\Big( \frac{1-q^{j_r-j_s}\xover x}{1-\xover x}\Big)
\big(1-ax_rx_sq^{j_r+j_s}\big)   \cr 
\smallprod n (1-ax_r^2q^{2j_r})  
\sqprod{n}\frac{\qrfac{\xover{x}y_s}{j_r} }{\qrfac{aqx_rx_sy_s}{j_r}} (y_1y_2\dots y_n)^\sumj 
 \powerq{j}{n}
%
\end{multline}
Note that  $g_{\k}(q^{\m})$ reduces to the entries  of the $C_n$ Bailey transform in the form $F_{\m\k}(a)$ given in \eqref{Cn-Bailey-F1}. It is analytic in the variables $\yvec$.

\begin{Theorem}[A multiple series extension of \eqref{wang-ma1}]\label{th:an-cntrans1}
Let $K(\y)$ and $\beta(\j)$ be chosen so that the series on both sides are analytic in $y_1, y_2, \dots, y_n$ and converge absolutely for each $y_r$ in a disk  around the origin; or, $y_r=q^{N_r}$, for $r=1, 2, \dots, n$, and so the series terminate.  Then
\begin{align}\label{an-cntrans1}
 K(\y)
\multisum{j}{n} & 
\vandermonde{x}{j}{n} \sqprod{n} \qrfac{\xover{x}y_s}{j_r}  
 \cr
 &
\smallprod n \frac{(1-bx_{r}q^{j_r+\sumj})}{(1-bx_r)\qrfac{bqx_ry_r}{\sumj}}
(y_1y_2\dots y_n)^{\sumj} 
\powerq{j}{n} \beta(\j)
\cr
=
\multisum{k}{n} &
\triprod n \frac{\Big(1-q^{k_r-k_s}x_r/x_s\Big)\Big(1-ax_rx_sq^{k_r+k_s}\Big)}{\big(1-x_r/x_s\big)(\big(1-ax_rx_s)}
\cr
&
\smallprod n \frac{1-ax_{r}^2q^{2k_r}}{1-ax_{r}^2}
\sqprod{n}\frac{\qrfac{\xover{x}y_s, ax_rx_s}{k_r}   }{\qrfac{q\xover{x}, aqx_rx_sy_s}{k_r}}
\cr
&
 \hspace{0.9cm}     (y_1y_2\dots y_n)^\sumk \powerq{k}{n}
 \cr
 &
\times 
\multsum{j}{k}{r}
\vandermonde{x}{j}{n} \smallprod n \frac{\qrfac{bqx_r}{j_r}}{ \qrfac{bx_r}{j_r+\sumj}} \cr
&
 \hspace{2cm}   \sqprod{n} 
\qrfac{ax_rx_sq^{k_s}, q^{-k_s}\xover{x}}{j_r}  
 (-1)^{\sumj}
q^{\suml_{r=1}^{n} rj_r + \binom{\sumj}{2}} \beta(\j)
\cr
&
\times
\underset{r=1, \dots, n}{\sum\limits_{0\leq m_r\leq k_r-j_r}}
\triprod n \frac{\Big(1-q^{j_r-j_s+m_r-m_s}x_r/x_s\Big)\Big(1-ax_rx_sq^{j_r+j_s+m_r+m_s}\Big)}
{\big(1-q^{j_r-j_s}x_r/x_s\big)\big(1-ax_rx_s\big)}
\cr
&
\hspace{3cm}  q^{\suml_{r=1}^n rm_r} K(q^{\j+\m})
\smallprod n \frac{ \qrfac{bx_rq^{1+j_r}}{m_r}}{ \qrfac{bx_rq^{1+j_r+\sumj}}{m_r}}
\cr
&
\hspace{3cm} \sqprod{n}\frac{\qrfac{q^{j_r-k_s}\xover{x}, ax_rx_sq^{k_s+j_r}}{m_r} }{\qrfac{q^{1+j_r-j_s}\xover{x}}{m_r} \qrfac{aqx_rx_s}{j_r+m_r} }
\end{align}
\end{Theorem}
\begin{Remark} The sum on the left hand side has terms associated with very-well-poised $A_n$ series. However, the 
outermost sum on the right hand side has terms associated with a very-well-poised $C_n$ series. 
\end{Remark}
\begin{proof}
The proof is analogous to that of Theorem~\ref{th:an-trans1}. 
Let $f_{\k}(\y;a)$ be given by \eqref{fk} and $g_{\j}(\y;a)$ by \eqref{gj}.
Referring to the exposition in \S\ref{sec:intro}, $g_{\k}(\y,a)$ is analogous to $f_k(y,a)$ in \eqref{Ay} and $f_{\k}(\y,b)$ is analogous to $h_k(y,b)$ in \eqref{Ayb}.

We begin with an expression of the form
\begin{equation*}
C(\y):= K(y)\multisum{j}{n} f_{\j}(\y;b) \beta(\j)
\end{equation*}
containing a parameter $b$, and try to write it 
in the form
$$C(\y)=\multisum{k}{n} g_{\k}(\y;a) \alpha(\k)$$
with parameter $a$. 
As before, we solve for $\alpha(\k)$ in terms of $C(q^{m})$, using Proposition~\ref{prop:cn-bailey-inversion}, and follow the 
steps in the proof of Theorem~\ref{th:an-trans1}. The algebraic simplifications are also similar. 
In particular, \eqref{elem1} is useful again. 
\end{proof}

Next we specialize $K(\y)$ in such a manner that the inner sum can be summed by using 
a $D_n$ $_3\phi_2$ summation. 

\begin{Theorem}[A multiple series extension of \eqref{wang-ma2}]\label{th:an-cntrans2}
Let $a$, $A$, $b$, $x_1,x_2,\dots, x_n$ and $\beta(\j)$ be such that the 
series on both sides are analytic in $y_1, y_2, \dots, y_n$ and converge absolutely for each $y_r$ in a disk  around the origin; or, $y_r=q^{N_r}$, for $r=1, 2, \dots, n$, and so the series terminate. 
\begin{align}\label{an-cntrans2}
\frac{ \qrfac{A}{\infty}}{ \qrfac{Ay_1y_2 \dots y_n}{\infty}} 
&
\smallprod n  \frac{  \qrfac{aqx_r^2,  bqx_ry_r, aAx_ry_r/b}{\infty}}
 { \qrfac{bqx_r,aA x_r/b}{\infty}}
 \cr
 &
\triprod n \qrfac{aqx_rx_s, aqx_rx_sy_ry_s }{\infty}  \sqprod{n} \frac{1}{\qrfac{aqx_rx_sy_s}{\infty}}
\cr
&
\multisum{j}{n} 
\vandermonde{x}{j}{n} \sqprod{n} \qrfac{\xover{x}y_s}{j_r}  
 \cr
 &
\hspace{1cm}   \smallprod n \frac{(1-bx_{r}q^{j_r+\sumj})}{(1-bx_r)\qrfac{bqx_ry_r}{\sumj}}
(y_1y_2\dots y_n)^{\sumj} 
\powerq{j}{n} \beta(\j)
\cr
&
\hspace{-0.5cm}
=
\multisum{k}{n} 
\triprod n \frac{\Big(1-q^{k_r-k_s}x_r/x_s\Big)\Big(1-ax_rx_sq^{k_r+k_s}\Big)}
{\Big(1-x_r/x_s\Big)\Big(1-ax_rx_s\Big)}
\cr
&
\smallprod n  \frac{(1-ax_{r}^2q^{2k_r}) \qrfac{ ax_r/b, bqx_r/A}{k_r}}
{(1-ax_r^2) \qrfac{bqx_r, aAx_r/b}{k_r}}
\sqprod{n}\frac{\qrfac{\xover{x}y_s, ax_rx_s }{k_r}  }{\qrfac{q\xover{x},aqx_rx_sy_s}{k_r}  }
\cr
&
 \hspace{0.9cm}     (Ay_1y_2\dots y_n)^\sumk \powerq{k}{n}
 \cr
 &
\times 
\multsum{j}{k}{r}
\triprod n \frac{\big(1-q^{j_r-j_s}x_r/x_s\big)}
{\big(1-x_r/x_s\big)\qrfac{ax_rx_s}{j_r+j_s}}
  \cr
&
 \hspace{2cm}   \sqprod{n} \qrfac{ax_rx_sq^{k_s}, q^{-k_s}\xover{x}}{j_r} 
\; \qrfac{A}{\sumj}  
 \cr
&
 \hspace{2cm}  
\smallprod n \frac{(1-bx_rq^{j_r+\sumj}) \qrfac{bq/ax_r}{\sumj-j_r}}
 {(1-bx_r) \qrfac{bq^{1-k_r}/ax_r, bx_rq^{1+k_r}}{\sumj}\qrfac{bqx_r/A}{j_r}}
\cr
&
 \hspace{2.5cm}   \Bigg( \frac{bq}{ aA} \Bigg)^{\sumj}
q^{\suml_{r=1}^{n} (r-1)j_r + \sum_{r<s} j_rj_s}  \smallprod n x_r^{-j_r} \beta(\j)
\end{align}
\end{Theorem}

\begin{proof} We choose $K(\y)$ in Theorem~\ref{th:an-cntrans1} so that the inner sum is summable using a $D_n$ balanced $_3\phi_2$ summation formula due to Milne and Lilly~\cite{ML1995} (see
\cite[Eq.\ 3.18]{GB1999a}):
\begin{align}\label{dn3p2-1}
\smallprod n \frac{\qrfac{cx_r/b}{N_r} \qrfac{aqx_r/c}{N_r} }{\qrfac{cx_r}{N_r}  \qrfac{abqx_r/c}{N_r}} & b^{\sumN} \cr
&  \hspace{-0.5cm}   =
\multsum{k}{N}{r} 
\vandermonde{x}{k}{n} 
\sqprod{n} \frac{\qrfac{q^{-N_s}x_r/x_s}{k_r}}{\qrfac{qx_r/x_s}{k_r}} \cr
&\hspace{1 cm}\times 
\triprod n \frac{1}{ \qrfac{ax_rx_s}{k_r+k_s}} 
  \sqprod{n} \qrfac{ax_rx_sq^{N_s}}{k_r} 
  \cr
  &
\hspace{1.5 cm} \times \smallprod n \frac{1}{\qrfac{cx_r, abqx_r/c}{k_r}} 
\qrfac{b}{\sumk}
q^{\suml_{r=1}^{n} rk_r } .
\end{align}

We consider \eqref{an-cntrans1} with 
\begin{align*}
K(\y)= \frac{ \qrfac{A}{\infty}}{ \qrfac{Ay_1y_2 \dots y_n}{\infty}} 
&
\smallprod n  \frac{  \qrfac{aqx_r^2,  bqx_ry_r, aAx_ry_r/b}{\infty}}
 { \qrfac{bqx_r,aA x_r/b}{\infty}}
 \cr
 &
\triprod n \qrfac{aqx_rx_s, aqx_rx_sy_ry_s }{\infty}  \sqprod{n} \frac{1}{\qrfac{aqx_rx_sy_s}{\infty}}.
\end{align*}
 
 After some elementary calculations using \cite[Appendix~1]{GR90}, we apply 
 the
 $$x_r\mapsto x_rq^{j_r}, N_r=k_r-j_r, a\mapsto a, b\mapsto Aq^{\sumj}
 , \text{ and } c\mapsto aA/b$$
 case of
 \eqref{dn3p2-1}
 to sum the inner-most sum on the right hand side of \eqref{an-cntrans1}. The resulting identity can be written as 
 \eqref{an-cntrans2}  after some further simplifications.
\end{proof}

Theorem~\ref{th:an-cntrans2} contains a $C_n$ non-terminating $_6\phi_5$ summation due to Gustafson~\cite{RG1989},  
see Lilly and Milne~\cite[Th.\  2.18]{LM1993}. It can be re-written to fit the notation used for $A_n$ series as follows.
 Let
$|{aq}/{bc_1c_2\dots c_n d}|<1$. Then
\begin{align}\label{cn-nt6p5}
\frac{\qrfac{ aq/bd}{\infty}}{\qrfac{aq/bc_1\dots c_n d}{\infty}}
  &
\smallprod n  \frac{\qrfac{aqx_r^2, aqx_r/bc_r,   aqx_r/c_r d }{\infty}}
 {\qrfac{ aqx_r/b, aqx_r/d}{\infty}} \cr
 &\hspace{-1cm}\times
 \triprod{n} \qrfac{aqx_rx_s, aqx_rx_s/c_rc_s}{\infty}
 \sqprod{n}\frac{1}{\qrfac{aqx_rx_s/c_s}{\infty}}\cr
=
\multisum{k}{n} &
\triprod{n} \frac{\big(1-q^{k_r-k_s}\xover x\big)\big(1-aq^{k_r+k_s}x_rx_s\big)}
{\big(1-\xover x\big)\big(1-ax_rx_s\big)}
\smallprod n \frac{1-aq^{2k_r}x_r^2}{1-ax_r^2}
\cr
& \times
\sqprod{n}\frac{\qrfac{ax_rx_s, c_s\xover{x}}{k_r} }{\qrfac{q\xover{x}, aqx_rx_s/c_s}{k_r}
 }
 \smallprod n \frac{\qrfac{bx_r,dx_r}{k_r}}
{\qrfac{aqx_r/b, aqx_r/d}{k_r}}
  \cr
& \hspace{1cm} \times
\Bigg(\frac{aq}{bc_1\dots c_n d}\Bigg)^{\sumk} q^{\suml_{r=1}^{n} (r-1)k_r}
.
\end{align}
To obtain \eqref{cn-nt6p5}, take
$$\beta_{j}=\prod_{r=1}^n \delta_{j_r,0}, b\mapsto a/b, y_r\mapsto 1/c_r, \text{ and  } A\mapsto aq/bd$$
in Theorem~\ref{th:an-cntrans2}. 

Theorem~\ref{th:an-cntrans2} is an expansion formula where the left hand side is an $A_n$ series, and the right hand side has a mix of $C_n$ and $D_n$ type factors.  The second expansion theorem of this section transforms a $C_n$ type series into a series which mixes elements of an $A_n$ series with a $D_n$ series. 
\begin{Theorem}[A multiple series extension of \eqref{wang-ma1}]\label{th:cn-antrans-3}
Let $K(\y)$ and $\beta(\j)$ be chosen so that the series on both sides are analytic in $y_1, y_2, \dots, y_n$ and converge absolutely for each $y_r$ in a disk  around the origin; or, $y_r=q^{N_r}$, for $r=1, 2, \dots, n$, and so the series terminate.  Then
\begin{align}\label{cn-antrans-3}
K(\y)
\multisum{j}{n} &
\triprod n \frac{\big(1-q^{j_r-j_s}\xover x\big)\big(1-bx_rx_sq^{j_r+j_s}\big) } 
{\big(1-\xover x\big)\big(1-bx_rx_s\big)}
\cr
& 
\smallprod n \frac{1-bx_r^2q^{2j_r}}{1-bx_r^2}
\sqprod{n}\frac{\qrfac{\xover{x}y_s}{j_r} }{\qrfac{bqx_rx_sy_s}{j_r}} (y_1y_2\dots y_n)^\sumj 
 \powerq{j}{n}
\beta(\j) \cr
=
\multisum{k}{n} &
\vandermonde{x}{k}{n} 
\sqprod{n}\frac{\qrfac{\xover{x}y_s}{k_r} }{\qrfac{q\xover{x}}{k_r} }\cr
&
\smallprod n \frac{(1-ax_rq^{k_r+\sumk})\qrfac{ax_r}{\sumk}}{(1-ax_r)\qrfac{aqx_ry_r}{\sumk}}
(y_1y_2\dots y_n)^\sumk \powerq{k}{n} \cr
&\times 
\multsum{j}{k}{r}
\triprod n 
\frac{(1-q^{j_r-j_s}\xover x)\big( 1- b x_rx_s q^{j_r+j_s}\big)}
{ (1-\xover x)\big( 1- b x_rx_s \big) 
} 
\cr
&\hspace{1.25cm}
\smallprod n \frac{\big(1-bx_r^2q^{2j_r}\big)\qrfac{ax_rq^{\sumk}}{j_r}}{\big(1-bx_r^2\big)}
 \cr
&\hspace{1.75cm}
\sqprod{n}\frac{ \qrfac{q^{-k_s}\xover{x}, bqx_rx_s}{j_r}}{\qrfac{bqx_rx_s}{j_r+j_s}}
(-1)^{\sumj} q^{\suml_{r=1}^{n} rj_r + \binom{\sumj}{2}} \beta(\j)\cr
&\times
\underset{r=1, \dots, n}{\sum\limits_{0\leq m_r\leq k_r-j_r}}
\triprod n \frac{1-q^{j_r-j_s+m_r-m_s} \xover {x} }{1-q^{j_r-j_s}  \xover{x}}
\cr
&\hspace{2.25cm}
\sqprod{n}\frac{\qrfac{q^{j_r-k_s}\xover{x}, bx_rx_sq^{1+j_r}}{m_r} 
}{\qrfac{q^{1+j_r-j_s}\xover{x}, bx_rx_sq^{1+j_r+j_s}}{m_r} }
\cr
&\hspace{2.5cm}
\smallprod n \frac{\qrfac{ax_rq^{j_r+\sumk}}{m_r}}
{\qrfac{aqx_r}{j_r+m_r}}
\;
q^{\suml_{r=1}^n rm_r} K(q^{\j+\m})
.
\end{align}
\end{Theorem}
\begin{Remark} The sum on the left hand side has terms associated with very-well-poised $C_n$ series. However, the 
outermost sum on the right hand side has terms associated with a very-well-poised $A_n$ series; the middle sums has elements of the $C_n$ series too. 
\end{Remark}
\begin{proof}
The proof is analogous to that of Theorem~\ref{th:an-trans1}.
Let $f_{\k}(\y;a)$ be given by \eqref{fk}. 
We begin with an expression of the form
\begin{equation*}
D(\y):= K(y)\multisum{j}{n} g_{\j}(\y;b) 
 \beta(\j)
\end{equation*}
containing a parameter $b$, and try to write it 
in the form
$$D(\y)=\multisum{k}{n} f_{\k}(\y;a) \alpha(\k)$$
with parameter $a$. 
As before, we can solve for $\alpha(\k)$ in terms of $D(q^{m})$, and follow the 
steps in the proof of Theorem~\ref{th:an-trans1}. We use the $A_n$ Bailey matrix inversion in Proposition~\ref{prop:bailey-inversion}.
\end{proof}


In the next result, we specialize $K(\y)$ in such a manner that the inner sum can be summed by using 
the $D_n$ $_3\phi_2$ summation due to the first author. 
We obtain a result for terminating series.
\begin{Theorem}[A terminating extension of \eqref{wang-ma2} over root  systems]\label{th:dn-result5}
\begin{align}\label{dn-result5}
\frac{1}{\qrfac{ aA/b}{\sumN}} &
\smallprod n  \qrfac{Ax_r,   aqx_r}{N_r}
 \triprod{n}\qrfac{bqx_rx_s}{N_r+N_s}
 \sqprod{n}\frac{1}{\qrfac{bqx_rx_s}{N_r}}\cr
\times
\multsum{j}{N}{r} &
\triprod n \frac{\big(1-q^{j_r-j_s}\xover x\big)\big(1-bx_rx_sq^{j_r+j_s}\big) }
{\big(1-\xover x\big)\big(1-bx_rx_s\big)}
\cr
&
\smallprod n \frac{1-bx_r^2q^{2j_r}}{1-bx_r^2}
\sqprod{n}\frac{\qrfac{q^{-N_s}\xover{x}}{j_r} }{\qrfac{bx_rx_sq^{1+N_s}}{j_r}} q^{\sumj\sumN+\suml_{r=1}^{n} (r-1)j_r} 
\beta(\j) \cr
=
\multsum{k}{N}{r} &
\triprod n \frac{\big(1-q^{k_r-k_s}\xover x\big)\qrfac{bqx_rx_s}{k_r+k_s}}
{\big(1-\xover x\big)}
\sqprod{n}\frac{\qrfac{q^{-N_s}\xover{x}}{k_r} }{\qrfac{q\xover{x}}{k_r}
\qrfac{bqx_rx_s}{k_r} }
\cr
&\smallprod n \frac{(1-ax_rq^{k_r+\sumk})\qrfac{ax_r}{\sumk}\qrfac{aq^{\sumk-k_r}/bx_r, bqx_r/A}{k_r}}
{(1-ax_r) \qrfac{ax_rq^{1+N_r}}{\sumk}}
  \cr
& 
\frac{1}{\qrfac{aA/b}{\sumk}}
\big(Aq^{\sumN}\big)^{\sumk}q^{\suml_{r=1}^{n} (r-1)k_r-\sum\limits_{r<s} k_rk_s} \smallprod n x_r^{k_r}
\cr
&\times 
\multsum{j}{k}{r}
\triprod n \frac{\big(1-q^{j_r-j_s}\xover x\big)\big(1-bx_rx_sq^{j_r+j_s}\big)}
{\big(1-\xover x\big)\big(1-bx_rx_s\big)}
 \cr
&\hspace{1.5cm}
\smallprod n \frac{\big(1-bx_r^2q^{2j_r}\big)\qrfac{ax_rq^{\sumk}, Ax_r}{j_r}}
{(1-bx_r^2)\qrfac{bqx_r/A, bx_rq^{1-\sumk}/a}{j_r}}
 \cr
&\hspace{2cm}
\sqprod{n}
\frac{\qrfac{q^{-k_s}\xover{x}}{j_r}} {\qrfac{bx_rx_sq^{1+k_s}}{j_r}}
\Big(\frac{bq}{aA}\Big)^{\sumj} q^{\suml_{r=1}^{n} (r-1)j_r} \beta(\j)
.
\end{align}
\end{Theorem}
\begin{Remark} Theorem~\ref{th:dn-result5} contains a $D_n$ terminating very-well-poised $_6\phi_5$ summation due to the first author~\cite[Th.\ 2]{GB1999a}. To see this,  take
$$\beta_{j}=\prod_{r=1}^n \delta_{j_r,0}, b\mapsto a/c, \text{ and  } A\mapsto aq/bc$$ 
in  \eqref{dn-result5}.
\end{Remark}
\begin{proof} We choose $K(\y)$ in Theorem~\ref{th:cn-antrans-3} so that the inner sum is summable using an $A_n$  terminating balanced $_3 \phi _2$ summation theorem given in~\cite[Th.~1]{GB1999a}:
\begin{align}\label{dn3p2}
&\hspace{-1.5cm}
\triprod n \qrfac{cx_rx_s}{N_r+N_s}\sqprod n \frac{1}{ \qrfac{cx_rx_s}{N_r}}
\frac
{ \smallprod n { \qrfac{cx_r/a , c x_r /b}{N_r}}}
{  \qrfac{c/ab}{\sumN}}
  \cr
= &
\multsum{m}{N}{r} 
\vandermonde{x}{m}{n} \triprod n \qrfac{cx_rx_s}{m_r+m_s}
\cr
& \hspace{1 cm}\times
\sqprod{n} \frac{\qrfac{q^{-N_s}x_r/x_s}{m_r}}{\qrfac{qx_r/x_s, cx_rx_s}{m_r}}  
 \frac{\smallprod n\qrfac{ax_r, bx_r}{m_r}}{\qrfac{abq^{1-\sumN}/c}{\summ}} \;
q^{\suml_{r=1}^{n} rm_r } .
\end{align}
We consider the terminating (i.e. $y_r=q^{N_r}$ for $r=1, 2, \dots, n$) case of \eqref{cn-antrans-3} with 
\begin{multline*}
K(\y)=
\frac{\qrfac{aAy_1\dots y_n/b}{\infty}}
 {\qrfac{ aA/b}{\infty}}
\smallprod n  \frac{\qrfac{Ax_r,   aqx_r}{\infty}}
 {\qrfac{ Ax_ry_r, aqx_ry_r}{\infty}}\\
 \triprod{n}\frac{\qrfac{bqx_rx_s}{\infty} }{\qrfac{bqx_rx_sy_ry_s}{\infty}} 
 \sqprod{n}\frac{\qrfac{bqx_rx_sy_r}{\infty} }{\qrfac{bqx_rx_s}{\infty}},
 \end{multline*}
 with $y_r=q^{N_r}$ for $r=1, 2, \dots, n.$
 After some elementary calculations using \cite[Appendix~1]{GR90}, we apply 
 the
 $$x_r\mapsto x_rq^{j_r}, N_r=k_r-j_r,  a\mapsto aq^{\sumK}, b\mapsto A
 , \text{ and } c\mapsto bq$$
 case of
 \eqref{dn3p2}
 to sum the resulting inner-most sum obtained on the right hand side of \eqref{cn-antrans-3}. The resulting identity can be written as \eqref{dn-result5}. 
\end{proof}

\section{Application: Expansions of infinite products}\label{sec:app1}
In this section we show how our results can be applied to give many different kinds of extensions of a result due to
Liu~\cite[Th.\ 1.3]{Liu2013a}. Liu's result is an extension of the non-terminating $_6\phi_5$ sum which expands a
product consisting of several infinite $q$-factorials into a double sum. Our objective is to illustrate how this can be done; for this, we write down one result that extends the $C_n$ non-terminating $_6\phi_5$ summation given in \eqref{cn-nt6p5} due to Gustafson~\cite{RG1989}.  
These kinds of results may be useful in number-theoretic contexts, where one wishes to expand powers of theta functions
as multiple series; see
Milne~\cite{Milne2002}, Bartlett and Warnaar~\cite{BW2015}, and Griffin, Ono and Warnaar~\cite{GOW2016} for examples of some deep theorems of this nature. For simpler examples, involving double series, see Andrews \cite[Th.\ 5]{Andrews1986b} and Liu~\cite{Liu2013a}. However, in this paper we have not specialized our results to search for such applications.

Liu's result can be stated as follows:
\begin{multline}\label{liu-app1}
y^{l}\frac{\qrfac{aq, bqy}{\infty}}{\qrfac{aqy, bq}{\infty}}
\smallprod m \frac{\qrfac{aA_r, aB_ry}{\infty}}{\qrfac{aA_ry, aB_r}{\infty}}\\
=\sum_{k=0}^\infty 
\frac{(1-a q^{2k})\qrfac{1/y, a}{k}}
{(1-a)\qrfac{q, aqy}{k}} y^k 
\rPhis{m+2}{m+1}{q^{-k}, aq^k, aA_1, aA_2, \dots, aA_m}{bq, aB_1, aB_2, \dots, aB_m}{q}{q^{l+1}},
\end{multline}
where we have used the $_r\phi_s$ notation from \cite[p.~4]{GR90}. We have taken 
$a\mapsto yq$, $\alpha=a$, $b\mapsto bq/a$, $b_r\mapsto B_r$ and $c_r\mapsto C_r$ in 
\cite[Th.\ 1.3]{Liu2013a}. It is assumed that the denominator parameters are such that the denominator is never $0$, and the series, if non-terminating, converges. 

In multiple series extensions of this result, we can introduce considerable variation in the way the parameters appear; thus we state our example in generic terms. 
\begin{Theorem}[An extension of the $C_n$ non-terminating $_6\phi_5$ summation]\label{th:cn-app1}
Let $H(\y)$ be chosen so that both sides are analytic in $y_1, y_2, \dots, y_n$ and converge absolutely for each $y_r$ in a disk  around the origin; or, $y_r=q^{N_r}$, for $r=1, 2, \dots, n$, and so the series terminates.  Then
\begin{align}\label{cn-app1}
\frac{ \qrfac{aA}{\infty}}{ \qrfac{aAy_1y_2 \dots y_n}{\infty}} 
&
\smallprod n  \frac{  \qrfac{aqx_r^2,  bqx_ry_r, aBx_ry_r}{\infty}}
 { \qrfac{bqx_r,aB x_r}{\infty}} 
 \cr
 &
\triprod n \qrfac{aqx_rx_s, aqx_rx_sy_ry_s }{\infty}  \sqprod{n} \frac{1}{\qrfac{aqx_rx_sy_s}{\infty}}
H(\y)
\cr
=
\multisum{k}{n} &
\triprod n \frac{\Big(1-q^{k_r-k_s}x_r/x_s\Big)\Big(1-ax_rx_sq^{k_r+k_s}\Big)}
{\big(1-x_r/x_s\big)\big(1-ax_rx_s\big)}
\cr
&
\smallprod n \frac{1-ax_{r}^2q^{2k_r}}{1-ax_r^2}
\sqprod{n}\frac{\qrfac{\xover{x}y_s}{k_r} \qrfac{ax_rx_s}{k_r}  }{\qrfac{q\xover{x}}{k_r} \qrfac{aqx_rx_sy_s}{k_r} }
\cr
&
 \hspace{0.9cm}     (y_1y_2\dots y_n)^\sumk \powerq{k}{n}
 \cr
 &
\times
\underset{r=1, \dots, n}{\sum\limits_{0\leq m_r\leq k_r}}
\triprod n \frac{\Big(1-q^{m_r-m_s}x_r/x_s\Big)}
{\big(1-x_r/x_s\big)\qrfac{ax_rx_s}{m_r+m_s}}
\cr
& \hspace{1.75cm}
 \sqprod{n}\frac{\qrfac{q^{-k_s}\xover{x}, ax_rx_sq^{k_s}}{m_r} }{\qrfac{q\xover{x}}{m_r} }
 \; \qrfac{aA}{\summ} 
 \cr
 &\hspace{2cm}
 \smallprod n \frac{1}{\qrfac{bqx_r,aBx_r}{m_r}}
 q^{\suml_{r=1}^n rm_r} H(q^{\m})
\end{align}
\end{Theorem}
\begin{proof} 
We consider \eqref{an-cntrans1} with $K(y)$ replaced by 
\begin{align*}
\frac{ \qrfac{aA}{\infty}}{ \qrfac{aAy_1y_2 \dots y_n}{\infty}} 
&
\smallprod n  \frac{  \qrfac{aqx_r^2,  bqx_ry_r, aBx_ry_r}{\infty}}
 { \qrfac{bqx_r,aB x_r}{\infty}} 
 \cr
 &
\triprod n \qrfac{aqx_rx_s, aqx_rx_sy_ry_s }{\infty}  \sqprod{n} \frac{1}{\qrfac{aqx_rx_sy_s}{\infty}}
H(\y),
\end{align*}
and 
$$\beta(\j) =\prod_{r=1}^n \delta_{j_r,0}.$$ 
After simplifying as in the proof of Theorem~\ref{th:an-cntrans2}, we obtain \eqref{cn-app1}. The only thing different from
Theorem~\ref{th:an-cntrans2} is that we don't apply a $_3\phi_2$ summation. 
\end{proof}

Depending on what kinds of products we require on the left hand side, we can take different values of $H(\y)$. For example, the following types of products have appeared in this theory, and may be useful in different contexts. 

\begin{align*}
\sqprod n & \frac{
\qrfac{aC\xover x, aDy_s\xover x}{\infty}}
{\qrfac{aCy_s\xover x, aD\xover x}{\infty}};
\sqprod n \frac{
\qrfac{aEx_rx_s, aFx_rx_sy_s}{\infty}}
{\qrfac{aEx_rx_sy_s, aFx_rx_s}{\infty}};\cr
\smallprod n & \frac{
\qrfac{aGx_r, aHx_ry_r}{\infty}}
{\qrfac{aGx_ry_r, aHx_r}{\infty}};
\frac{\qrfac{aJ, aKy_1y_2\dots y_n}{\infty}}{\qrfac{aJy_1y_2\dots y_n, aK, }{\infty}}; \;
\smallprod n y_r^{l_r}.
\end{align*}

This is a very small sample of possible products. We can take a mix of multiple products of any type. The only requirement is that the $H(\y)$ is analytic in a region of the form $|y_r|<\epsilon_r$, for $r=1, 2, \dots, n$. In addition, see Remark~\ref{note-convergence}.

\section{An $A_n$ extension of Liu's third expansion formula}\label{sec:liu3}
In this section we give an $A_n$ extension of \eqref{liu3}. Liu used \eqref{liu3} to find many formulas, including one transformation formula that can be regarded as an extension of Watson's transformation
\cite[Eq.\ (2.5.1)]{GR90}. Our intent here is to illustrate another application of the technique used in this paper.

First, we codify our approach to Wang-Ma's results given in \S\ref{sec:intro} in the form of the following Lemma. 
\begin{Lemma}\label{lemma:qLidea} Let $\F=(F_{km})$ and $\G=(G_{km})$ be a pair of inverse, infinite, lower triangular matrices. Let
$f_k(y)$ be analytic in $y$ in a disk around the origin, such that $f_k(q^N)=F_{Nk}$. Let $K(y)$, $h_j(y)$ and $\beta(j)$ 
be such that both sides of the following transformation formula are analytic in $y$ in a disk $|y|<\epsilon$; or $y=q^N$ for some non-negative integer $N$,  $h_j(q^N)=0=f_j(q^N)$ for $j>N$, and the series terminate.
Then
\begin{equation}\label{lemma-idea}
 K(y)\sum_{j=0}^\infty h_j(y) \beta_j
 = 
 \sum_{k=0}^\infty f_k(y) \sum_{j=0}^k \beta_j \sum_{m=0}^{k-j} h_j(q^{j+m}) G_{k,j+m} K(q^{j+m})
\end{equation}
\end{Lemma}
\begin{Remark} The key idea in this paper is to use equivalent forms of the Bailey matrices given in \eqref{BaileyF} and
\eqref{BaileyG}; further, the choice of  $K(y)$ that appears to be useful is when the inner sum on the right is summable using the terminating, balanced, $_3\phi_2$ summation theorem. We expect to examine other matrices and summation results in a separate work. 
\end{Remark}

We now give an alternate proof of Liu's Theorem~9.2 \cite{Liu2013c}. Liu used $q$-Calculus but we use $q$-Lagrange inversion as given in Lemma~\ref{lemma:qLidea}. 
\begin{proof}[Proof of \eqref{liu3}]
We use the terminating case $y=q^N$ of \eqref{lemma-idea} with
$$K(y) =
\frac{\qrfac{ aq, aAB/q, aAy, aBy}{\infty}}{\qrfac{ aqy, aABy/q, aA, aB}{\infty}} ;
h_{j}(y) = 
\frac{\qrfac{1/y}{j}}{\qrfac{q^2/aABy}{j}};
\beta(j) =
\qrfac{q/A,q/B}{j} q^{j}A_{j};
$$
$f_{k}(y)=f_k(y;a)$ from \eqref{fk-1}; and 
$G_{km}=G_{km}(a)$ from \eqref{BaileyG}. 
After some simplification, we apply a special case of the $q$-Pfaff-Saalch{\"u}tz summation \cite[Eq.\ (1.7.2)]{GR90}
to obtain \eqref{liu3}.
\end{proof}

Extending this proof to prove multiple series extensions of \eqref{liu3} over root system is routine, given our work in this paper. As in \S\S\ref{sec:trans1} and \ref{sec:cn-dn}, there are several different extensions. We present one as a sample. The computation is a mild variation of the proofs of Theorems~\ref{th:an-trans1} and \ref{th:an-result5a}.

\begin{Theorem}[An $A_n$ extension of \eqref{liu3}]\label{th:an-liu3} Let ${N_r}$, for $r=1, 2, \dots, n$, be non-negative integers, and $A_{\j}$ and the parameters $a$, $A$, $B$, and $x_1,x_2,\dots, x_n$ and be such that the series below have no zero terms in the denominator.  
Then we have
\begin{align}\label{an-liu3}
\frac{\qrfac{aAB/q}{\sumN}}{\qrfac{ aA}{\sumN}} &
\smallprod n \frac{\qrfac{aqx_r}{N_r}}{\qrfac{aBx_r}{N_r}} \cr
\multsum{j}{N}{r} &
\vandermonde{x}{j}{n}  
\sqprod{n}\qrfac{q^{-N_s}\xover{x}}{j_r} \cr
&\smallprod n \qrfac{qx_r/A}{j_r}
 \frac{\qrfac{q/B}{\sumj}}{\qrfac{q^{2-\sumN}/aAB}{\sumj}}
q^{\suml_{r=1}^{n} rj_r }  A_{\j} \cr
=
\multsum{k}{N}{r} &
\vandermonde{x}{k}{n} 
\sqprod{n}\frac{\qrfac{q^{-N_s}\xover{x}}{k_r} }{\qrfac{q\xover{x}}{k_r} }\cr
&
\smallprod n \frac{\big(1-ax_rq^{k_r+\sumk}\big)\qrfac{ax_r}{\sumk} \qrfac{qx_r/A}{k_r}}
{(1-ax_r)\qrfac{ax_rq^{1+N_r}}{\sumk}\qrfac{aBx_r}{k_r}} \cr
& \frac{\qrfac{q/B}{\sumk}}{\qrfac{aA}{\sumk}}
(aABq^{\sumN-1})^\sumk \powerq{k}{n} \cr
&\times 
\multsum{j}{k}{r}
\vandermonde{x}{j}{n} 
\sqprod{n}\qrfac{q^{-k_s}\xover{x}}{j_r} \cr
&\hspace{2cm}
\smallprod n\qrfac{ax_rq^{\sumk}}{j_r}
q^{\suml_{r=1}^{n} rj_r } A_{\j}
.
\end{align}
\end{Theorem}
\begin{Remarks}\
\begin{enumerate}
\item We have obtained a result for terminating series. In certain situations, we can obtain a non-terminating result where we can replace $q^{N_r}$ by $y_r$ just as is done in Theorem~\ref{th:an-result5a}. For example, when $B\to 0$, then we obtain an analytic function in $y_1, y_2, \dots, y_n$ and we can use the $B\to 0$ case of
\eqref{an-liu3} to obtain a non-terminating expansion formula. 
\item
 Liu~\cite{Liu2013c} has given many examples and special cases of his result; we can do analogous computations to obtain many further applications of our approach. 
\end{enumerate}
\end{Remarks}
\begin{proof}
We use a straightforward extension of the terminating Lemma~\ref{lemma:qLidea}  to multiple series with 
 $y_r=q^{N_r}$ for $r=1$, $\dots$, $n$, and 
\begin{align*}
K(\y)&=
\frac{\qrfac{aAB/q, aAy_1\dots y_n}{\infty}}{\qrfac{aABy_1\dots y_n/q, aA}{\infty}} 
\smallprod n \frac{\qrfac{aqx_r, aBx_ry_r}{\infty}}{\qrfac{aqx_ry_r, aBx_r}{\infty}} ;\\
h_{\j}(\y) &= 
 \vandermonde{x}{j}{n}  
\sqprod{n}\qrfac{\xover{x}y_s}{j_r}
\frac{1}{\qrfac{q^2/aABy_1\dots y_n}{\sumj}}
\powerq{j}{n};\\
\beta(\j)&=
\smallprod n \qrfac{qx_r/A}{j_r} \;
 \qrfac{q/B}{\sumj} q^{\sumj}A_{\j}.
 \end{align*}
Along with the above, we need 
$f_{\k}(\y;a)$ from \eqref{fk} and 
$G_{\k\m}(a)$ from \eqref{An-Bailey-G}. After inserting these in a suitable generalization of \eqref{lemma:qLidea},
we apply the 
 $$x_r\mapsto x_rq^{j_r}, N_r=k_r-j_r, a\mapsto aq^{\sumK}, b\mapsto aAB/q
 , \text{ and } c\mapsto aB$$
 case of
 \eqref{an3p2-1} (due to Milne~\cite[Theorem 4.1]{Milne1997}) and simplify to obtain \eqref{an-liu3}.
\end{proof}

\section{Concluding remarks}
In this paper, we have seen several multiple series extensions of 
\eqref{wang-ma1} and \eqref{wang-ma2}, which, in turn, extend Liu's useful expansion formulas
\eqref{liu-main1} and \eqref{liu-gen1}. The method we use, essentially due to Wang and Ma \cite{WM2017}, also applies to prove Liu's third formula \eqref{liu3}. As we saw, this approach brings Liu's expansion formulas within the Bailey transform methodology. 

We listed three extensions of \eqref{wang-ma1} and four of \eqref{wang-ma2}; more are possible. Our extensions contain several $_6\phi_5$ summation theorems over root systems. 
There are additional summations in \cite{Milne1997} and \cite{MS2008} which hint towards additional results of this nature.  There are further interesting special cases. For example, note that the sums on the left hand side of
\eqref{an-result5a} and \eqref{an-cntrans2} are identical. Comparing the two results gives us a transformation formula. In addition, there are many possible applications and useful special cases 
in Liu's papers that can be extended to multiple series.  We expect to explore these ideas later.

\subsection*{Acknowledgements} This work was initiated during the first author's visit to the School of Physical Sciences (SPS), JNU, Delhi. Surbhi Rai was supported by the Ministry of Education, Govt.\ of India. This paper is part of her Ph.D. thesis.  We thank Bruce Berndt, Xinrong Ma and the anonymous referee for helpful advice.  We thank Amitabha Tripathi for 
his help during this project. 

%

\end{document}